\documentclass[11pt]{article}
\usepackage{mathrsfs}
\usepackage{bbm}

\usepackage{verbatim}

\usepackage{amsthm}
\usepackage{amsmath}
\usepackage{amssymb}

\usepackage{graphicx}
\usepackage{epstopdf}
\usepackage{titling}

\usepackage{enumerate}

\usepackage{algpseudocode}
\usepackage{algorithm}
\usepackage{algorithmicx}
\usepackage{tikz}
\usepackage{appendix}

\newtheorem{theorem}{Theorem}[section]

\newtheorem{lemma}[theorem]{Lemma}

\newtheorem{conjecture}[theorem]{Conjecture}
\newtheorem{definition}[theorem]{Definition}

\begin{document}

\title{\vspace{-0.85cm}The minimum number of nonnegative edges in hypergraphs}

\author{Hao Huang\thanks{Institute for Advanced Study, Princeton, NJ 08540 and DIMACS at Rutgers University. Email: {\tt huanghao@math.ias.edu}. Research supported in part by NSF grant DMS-1128155.}  \and Benny Sudakov
\thanks{Department of Mathematics, ETH, 8092 Zurich, Switzerland. Email: {\tt 
bsudakov@math.ucla.edu.} Research supported in part by SNSF grant 200021-149111 and by a USA-Israel BSF grant.}}
\date{}
\maketitle
\setcounter{page}{1}
\vspace{-2em}
\begin{abstract}
An $r$-unform $n$-vertex hypergraph $H$ is said to have the Manickam-Mikl\'os-Singhi (MMS) property if for every assignment of weights to its vertices with nonnegative sum, the number of edges whose total weight is nonnegative is at least the minimum degree of $H$.
In this paper we show that for $n>10r^3$, every $r$-uniform $n$-vertex hypergraph with equal codegrees has the MMS property, and the bound on $n$ is essentially tight up to a constant factor. This result has two immediate corollaries. First it shows that every set of $n>10k^3$ real numbers with nonnegative sum has at least $\binom{n-1}{k-1}$ nonnegative $k$-sums, verifying the Manickam-Mikl\'os-Singhi conjecture for this range.
More importantly, it implies the vector space Manickam-Mikl\'os-Singhi conjecture which states that for $n \ge 4k$ and any weighting on the $1$-dimensional subspaces of $\mathbb{F}_{q}^n$ with nonnegative sum, the number of nonnegative $k$-dimensional subspaces is at least ${n-1 \brack k-1}_q$. We also discuss two additional generalizations, which can be regarded as analogues of the Erd\H{o}s-Ko-Rado theorem on $k$-intersecting families.
\end{abstract}
\section{Introduction}
Given an $r$-uniform $n$-vertex hypergraph $H$ with minimum degree $\delta(H)$, suppose every vertex has a weight $w_i$ such that $w_1 + \cdots + w_n \ge 0$.
How many nonnegative edges must $H$ have? An edge of $H$ is nonnegative if the sum of the weights on its vertices is $\geq 0$.
Let $e^{+}(H)$ be the number of such edges. By assigning weight $n-1$ to the vertex with minimum degree, and $-1$ to the remaining vertices, it is easy to see that the number of nonnegative edges can be at most $\delta(H)$. It is a very natural question to determine when this easy
upper bound is tight, which leads us to the following definition.

\begin{definition}
A hypergraph $H$ with minimum degree $\delta(H)$ has the \textit{MMS property} if for every weighting $w: V(H) \rightarrow \mathbb{R}$ satisfying
$\sum_{x \in v(H)} w(x) \ge 0$, the number of nonnegative edges is at least $\delta(H)$.
\end{definition}

The question, which hypergraphs have MMS property, was motivated by two old conjectures of Manickam, Mikl\'os, and Singhi \cite{manickam-miklos, manickam-singhi}, both of which were raised in their study of so-called first distribution invariant of certain association schemes.
\begin{conjecture}\label{conj_mms}
Suppose $n \ge 4k$, and we have $n$ real numbers $w_1, \cdots, w_n$ such that $w_1 + \cdots + w_n \ge 0$, then there are at least
$\binom{n-1}{k-1}$ subsets $A$ of size $k$ satisfying $\sum_{w_i \in A} w_i \ge 0$.
\end{conjecture}

The second conjecture is an analogue of Conjecture \ref{conj_mms} for vector spaces. Let $V$ be a $n$-dimensional vector space over a finite field $\mathbb{F}_q$. Denote by ${V \brack k}$ the family of $k$-dimensional subspaces of $V$, and the $q$-Gaussian binomial coefficient ${n \brack k}_q$ is defined as $\prod_{0 \le i < k} \frac{q^{n-i}-1}{q^{k-i}-1}.$

\begin{conjecture} \label{conj_mms_vector}
Suppose $n \ge 4k$, and let $w: {V \brack 1} \rightarrow \mathbb{R}$ be a weighting on the one-dimensional subspaces of $V$ such that $\sum_{v \in {V \brack 1}} w(v)=0$, then the number of $k$-dimensional subspaces $S$ with
$\sum_{v \in {V \brack 1}, v \subset S} w(v) \ge 0$ is at least ${n-1 \brack k-1}_q.$
\end{conjecture}

Conjecture \ref{conj_mms} can be regarded as an analogue of the famous Erd\H os-Ko-Rado theorem \cite{erdos-ko-rado}. The latter says that for $n \geq 2k$, a family of $k$-subsets of $[n]$ with the property that every two subsets have a nonempty intersection has size at most $\binom{n-1}{k-1}$. In both problems, the extremal examples correspond to a star, which consists of subsets containing a particular element in $[n]$. The Manickam-Mikl\'os-Singhi conjecture has been open for more than two decades, and various partial results were proven. There are several works verifying the conjecture for small $k$ \cite{hartke-stolee, manickam, marino-chiaselotti}. But most of the research focus on proving the conjecture for every $n$ greater than a given function $f(k)$.
Manickam and Mikl\'{o}s \cite{manickam-miklos} verified the conjecture for $n \ge (k-1)(k^k+k^2)+k$.
Later Tyomkyn \cite{tyomkyn} improved this bound to $n \ge e^{ck \log \log k}$. Alon, Huang, and Sudakov \cite{alon-huang-sudakov} obtained the
first polynomial bound $n>33k^2$. Later, Frankl \cite{frankl} gave a shorter proof for a cubic range $n\ge \frac{3}{2}k^3$.
A linear bound $n \ge 10^{46} k$ was obtained by Pokrovskiy \cite{pokrovskiy}. He reduced the conjecture to finding a $k$-uniform
hypergraph on $n$ vertices satisfying the MMS property (similar techniques were also employed earlier in \cite{manickam-miklos}).
The second conjecture, Conjecture \ref{conj_mms_vector}, was very recently proved by Chowdhury, Sarkis, and Shahriari \cite{chowdhury_vector}
simultaneously with our work. They also proved a quadratic bound $n \geq 8k^2$ for sets.

We observe that both conjectures can be reduced to proving that certain hypergraph has the MMS property. For the first conjecture,
simply let the hypergraph $H_1$ be the complete $k$-uniform hypergraph on $n$ vertices. For the second conjecture,
one can take the ${k \brack 1}_q$-uniform hypergraph $H_2$ with vertex set ${V \brack 1}$ and let edges correspond to $k$-dimensional subspaces.
Both hypergraphs are regular, and moreover the codegree of every pair of vertices is the same. It is tempting to conjecture that all such hypergraphs
satisfy the MMS property. The requirement that all the codegrees are equal may not be dropped. For instance, the tight Hamiltonian cycle
(the edges are consecutive $r$-tuples modulo $n$) when $n \equiv 1 \pmod r$ is not MMS. This can be seen by choosing the weights $w(xr+1)=n$ for $x=0, \cdots, n/r$
and all the other weights to be $-\frac{n+r}{r-1}$, which results in only $r-1$ nonnegative edges, as opposed to the fact that the degree is $r$.
Our main theorem indeed confirms that equal codegrees imply the MMS property.

\begin{theorem} \label{main_theorem}
Let $H$ be an $r$-uniform $n$-vertex hypergraph with $n>10 r^3$ and all the codegrees equal to $\lambda$. Then for every
weighting $w: V(H) \rightarrow \mathbb{R}$ with $\sum_v w_v \ge 0$, we have $e^{+}(H) \ge \delta(H)$. Moreover in the case of equality, all nonnegative edges
form a star, i.e., contain a fixed vertex of $H$.
\end{theorem}

The lower bound on $n$ in this theorem is tight up to a constant factor.  Our result immediately implies two corollaries. First it verifies Conjecture 
\ref{conj_mms} for a weaker range $n \ge \Omega(k^3)$. Moreover it also provides a proof of Conjecture \ref{conj_mms_vector}.

As mentioned earlier, there are some subtle connections between Manickam-Mikl\'os-Singhi conjecture and the Erd\H os-Ko-Rado theorem on intersecting families. 
In \cite{erdos-ko-rado}, Erd\H os, Ko and Rado also initiated the study of $k$-intersecting families (any two subsets have at least $k$ common elements). They 
show that for $k<t$, there exists an integer $n_0(k, t)$ such that for all $n \ge n_0(k, t)$ the largest $k$-intersecting family of $t$-sets are the 
$k$-stars, which are of size $\binom{n-k}{t-k}$. This result is equivalent to saying that in the $\binom{t}{k}$-uniform hypergraph $H$ whose vertices are 
$k$-subsets of $[n]$ and edges correspond to $t$-subsets of $[n]$, the maximum intersecting sub-hypergraph has size $\binom{n-k}{t-k}$. The following theorem 
says that for large $n$, this hypergraph has the MMS property. Note that this is not implied by Theorem \ref{main_theorem}, because the codegree of two 
vertices (as $k$-subsets) depends on the size of their intersection.

\begin{theorem}\label{theorem_k-tuple}
Let $k, t$ be positive integers with $t>k$, $n>Ct^{3k+3}$ for sufficiently large $C$ and let
$\{w_{X}\}_{X \in \binom{[n]}{k}}$ be a weight assignment with $\sum_{X \in \binom{ [n]}{k}} w_{X} \ge 0$. Then there are always at least $\binom{n-k}{t-k}$ subsets $T$ of size $t$ such that $\sum_{X \subset T} w_X \ge 0$.
\end{theorem}
This result can be regarded as an analogue of the $k$-intersecting version of the Erd\H os-Ko-Rado theorem. Moreover,
the Manickam-Mikl\'os-Singhi conjecture is a special case of this theorem corresponding to $k=1$ and $t=r$. Using a similar proof
one can also obtain a generalization of the vector space version of Manickam-Mikl\'os-Singhi conjecture.

\begin{theorem}\label{theorem_k-tuple_vector}
Let $k, t$ be positive integers with $t>k$, $n>Ck(t-k)$ for sufficiently large $C$, $V$ be the $n$-dimensional vector space over $\mathbb{F}_q$ and
let $\{w_{X}\}_{X \in {V \brack k}}$ be a weight assignment with $\sum_{X \in {V \brack k}} w_{X} \ge 0$. Then there are always at least ${n-k \brack t-k}_q$ $t$-dimensional subspaces $T$ such that $\sum_{X \in {V \brack k}, X \subset T} w_X \ge 0$.
\end{theorem}

The rest of the paper is organized as follows. In Section \ref{section_mainproof} we prove Theorem \ref{main_theorem} and deduce Conjecture \ref{conj_mms_vector} as a corollary. In Section \ref{section_tightness} we will have two constructions, showing that
the $n>\Omega(k^3)$ bound for Theorem \ref{main_theorem} is essentially tight.
The ideas presented in Section 2 are not enough to prove Theorem \ref{theorem_k-tuple}. Hence, in Section \ref{section_kint} we develop more sophisticated techniques to prove this theorem. We also sketch the lemmas needed to obtain Theorem \ref{theorem_k-tuple_vector} and leave the proof details to the appendix. The final section contains some open problems and
further research directions.

\section{Equal codegrees and MMS property} \label{section_mainproof}
In this section we prove Theorem \ref{main_theorem}. Without loss of generality, we may assume that $V(H)=[n]$, and the weights are $1= w_1 \ge w_2 \ge \cdots \ge w_n$, such that $\sum_{i=1}^n w_i = 0$. Suppose the number of edges in $H$ is $e$. By double counting, we have
that $H$ is $d$-regular with $d=\frac{n-1}{r-1}\lambda$ and that $dn=re$.
Consider the $2r$-th largest weight $w_{2r}$, we will verify Theorem \ref{main_theorem} for the following three cases respectively: (i) $w_{2r} \le  \frac{1}{2r^2}$; (ii) $w_{2r} \ge \frac{1}{2r}$; and (iii) $\frac{1}{2r^2} \le w_{2r} \le \frac{1}{2r}$.
\medskip

\begin{lemma}\label{lemma_case1}
If $w_{2r} \le \frac{1}{2r^2}$, then $e^{+}(H) \ge d$.
\end{lemma}
\begin{proof}
First we show that among the $d$ edges containing $w_1$, the number of negative edges is at most $\frac{5d}{6r}$. Denote the negative edges by $e_1, \cdots, e_m$ and the nonnegative edges by $e_{m+1}, \cdots, e_d$. By the definition of a negative edge, for every $1 \le i \le m$ we have
$$\sum_{j \in e_i \backslash \{1\}} w_j < -w_1 = -1.$$ Summing these inequalities, we get
$$\sum_{i=1}^m \sum_{j \in e_i \backslash \{1\}} w_j < -m.$$
Now we consider the sum $\sum_{i=m+1}^d \sum_{j \in e_i \backslash \{1\}} w_j$ and rewrite it as $\sum_{j} \alpha_j w_j$. The sum of coefficients $\alpha_j$'s is equal to $(d-m)(r-1)$. Note that in this sum $w_2, \cdots, w_{2r}$ each appears
at most $\lambda$ times (their codegree with $\{1\}$) and they are bounded by $1$, so in total they contribute no more than $2r \lambda$. The remaining variables $w_{2r+1}, \cdots, w_n$ contribute less than $(d-m)(r-1) w_{2r} < \frac{d-m}{2r}$. Combining these three estimates, we obtain that
$$\sum_{1 \in e} \sum_{j \in e \backslash \{1\}} w_j < -m+ 2r \lambda + \frac{d-m}{2r}.$$
By double counting, the left hand side is equal to $\lambda(w_2 + \cdots + w_n) = -\lambda$. Comparing these two quantities and doing simple calculations we get
$$m < 2r \lambda + \frac{d}{2r+1}<\frac{d}{5r}+\frac{d}{2r+1}<\frac{5d}{6r}.$$
Here we used that $n>10r^3$ and $\lambda = \frac{r-1}{n-1}d<d/(10r^2)$. Therefore we may assume there are at least $(1-\frac{5}{6r})d$ nonnegative edges through $w_1$.

If every edge through $w_1$ is positive, then we are already done; so we assume that there exists a negative edge $e$ through $w_1$. Suppose $w_u$ is the largest positive weight of vertex not contained in $e$. Such $u$ exists since otherwise $\sum_{i=1}^n w_i<0$, so we may assume that $u \ge 2$ and $\{1, \cdots, u-1\} \subset e$.
We claim that in this case there are many nonnegative edges through $w_u$ which are disjoint from $e$. Consider the set $S$ of $r$-tuples consisting of  all the edges through $w_u$ which are disjoint from $e$. Since each of the $r$ vertices of $e$ has at most $\lambda$ common neighbors with $w_u$, we have
$|S| \ge d - r  \lambda.$ Denote by $S^{-}$ the set of negative edges in $S$ and consider the sum
$\sum_{f \in S^{-}} \sum_{j \in f \backslash \{u\}} w_j$. Obviously it is at most $- w_u|S^{-}|$. Rewrite this sum as $\lambda \cdot ( \sum_{j \not \in e \cup \{u\}} \alpha_j w_j )$. Since all codegrees are $\lambda$, $\alpha_j \in [0,1]$ and $\sum_{j \not \in e \cup \{u\}} \alpha_j = (r-1)|S^{-}|/\lambda$, which implies that $\sum_{j \not \in e \cup \{u\}} (1-\alpha_j) = (n-r-1)-(r-1)|S^{-}|/\lambda$. Therefore
\begin{align*}
-w_u &<  \sum_{j \not \in e \cup \{u\}} w_j = \sum_{j \not \in e \cup \{u\}} \alpha_j w_j + \sum_{j \not\in e \cup \{u\}} (1-\alpha_j) w_j\\
&< -\frac{|S^{-}|}{\lambda} \cdot w_u + (n-r-1-(r-1)|S^{-}|/\lambda) \cdot w_u,
\end{align*}
The first inequality uses that the sum of all the weights is zero and that $e$ is a negative edge, so $\sum_{j \in e} w_j <0$. To see the second inequality, just observe that $w_j \le w_u$ for every $j \not \in e \cup \{u\}$.

By simplifying the last inequality we get $\frac{|S^{-}|}{\lambda} < \frac{n-r}{r}$. Therefore the number of nonnegative edges containing $w_2$ that are disjoint from $e$ is at least
$$|S|-|S^{-}| > d-r\lambda-\frac{n-r}{r} \lambda = \frac{n-(r^3-2r^2+2r)}{r(n-1)} d,$$
which is greater than $\frac{5d}{6r}$ if $n>10r^3$. These nonnegative edges, together with the $(1-\frac{5}{6r})d$ nonnegative edges through $w_1$, already give more than  $d$ nonnegative edges.
\end{proof}

\begin{lemma}\label{lemma_case2}
If $w_{2r} \ge \frac{1}{2r}$, then $e^{+}(H) \ge d$.
\end{lemma}
\begin{proof}
First we claim for any $1 \le i \le 2r$, there are at least $\frac{3}{5r}d$ nonnegative edges containing $w_i$. Let $S_i$ be the set of
negative edges containing $w_i$, then for any edge $e \in S_i$, $\sum_{j \in e \backslash \{i\}} w_j < -w_i.$
Summing up these inequalities, we have
$$\sum_{e \in S_i} \sum_{j \in e \backslash \{i\}} w_j < -|S_i|w_i.$$
Like the previous case, suppose the left hand side can be rewritten as $\lambda \cdot \sum_{j \ne i} \alpha_j w_j$, then $\alpha_j \in [0,1]$, and $\sum_{j \ne i} \alpha_j = (r-1)|S_i|/\lambda$, which implies $\sum_{j \ne i} (1-\alpha_j) = n-1 - (r-1)|S_i|/\lambda$. Since $\sum_{j \ne i} w_j = -w_i$, we have
\begin{align*}
-w_i &= \sum_{j \ne i} \alpha_j w_j + \sum_{j \ne i} (1-\alpha_j) w_j\\
&< -\frac{|S_i|}{\lambda} w_i + \sum_{1 \le j \le 2r} w_j + w_i \cdot \sum_{j>2r} (1-\alpha_j) \\
&\le -\frac{|S_i|}{\lambda} w_i + 2r w_1 + \Big(n-1-(r-1)\frac{|S_i|}{\lambda}\Big) w_i
\end{align*}
Substituting $\lambda=\frac{r-1}{n-1}d$ and $w_i \ge w_{2r} \ge \frac{1}{2r}$, gives
$$|S_i| \le (n+4r^2) \frac{\lambda}{r} = \frac{(r-1)(n+4r^2)}{r(n-1)}d,$$
which is less than $(1-\frac{3}{5r})d$ when $n>10r^3$. So there are at least $\frac{3}{5r}d$ nonnegative edges containing $w_i$. Note that for $1 \le i < j \le 2r$, $w_i$ and $w_j$ are simultaneously contained in at most $\lambda$ edges. Therefore the total number of nonnegative edges is at least
$$2r \cdot \frac{3}{5r}d - \binom{2r}{2} \cdot \lambda \ge \frac{6}{5}d- \frac{2r^2(r-1)}{n-1} d.$$
When $n>10r^3$, this gives more than $d$ nonnegative edges.
\end{proof}

\begin{lemma}\label{lemma_case3}
If $\frac{1}{2r^2} \le w_{2r} \le \frac{1}{2r}$, then $e^{+}(H) \ge d$.
\end{lemma}
\begin{proof}
Let $t$ be the index such that $w_t \ge 2r w_{2r}$ and $w_{t+1} < 2r w_{2r}$. Since $w_1 = 1 \ge 2r w_{2r}$ such $t$ exists and is between $1$ and $2r$.
For arbitrary $1 \le i \le t$, let $T_i$ be the set of negative edges containing $w_i$. Similarly as before, we assume
$$\sum_{e \in T_i} \sum_{j \in e \backslash \{i\}} w_j = \lambda \cdot \sum_{j \ne i} \alpha_j w_j.$$
Then $\sum_{j \ne i} \alpha_j w_j < -w_i |T_i|/\lambda$, and $\sum_{j \ne i} (1-\alpha_j) = n-1-(r-1)|T_i|/\lambda.$ We also have
\begin{align*}
-w_i &= \sum_{j \ne i} w_j=\sum_{j \ne i} \alpha_j w_j + \sum_{j \ne i} (1-\alpha_j) w_j\\
&< -\frac{|T_i|}{\lambda} w_i + \sum_{1 \le j \le t} (1-\alpha_j) w_j +  \sum_{t < j \le 2r} (1-\alpha_j)w_j + \sum_{j> 2r} (1-\alpha_j)w_j\\
&\le -\frac{|T_i|}{\lambda} w_i + t + (2r-t)w_t + (n-1-(r-1)|T_i|/\lambda) w_{2r} \\
&\le -\frac{|T_i|}{\lambda} w_i + t + (2r-t)w_t + (n-1-(r-1)|T_i|/\lambda) \frac{w_t}{2r}.
\end{align*}
Suppose $|T_i|/\lambda \geq 1$. Since $w_i \ge w_t$, we then have
$$\left( \frac{|T_i|}{\lambda}-1-(2r-t) - \frac{n-1-(r-1)|T_i|/\lambda}{2r}\right)w_t \le t \le t \cdot 2r^2w_{2r} \le     rt w_t.$$
Using that $n>10r^3$, $t \leq 2r$ and $\lambda=\frac{r-1}{n-1}d$ and rearranging the last inequality gives
$$
|T_i| \le \frac{2r}{3r-1} \left( \frac{n-1}{2r}+(r-1)t + 2r+1 \right) \lambda \le \frac{2r}{3r-1} \left( \frac{n-1}{2r}+2r^2+1 \right) \frac{r-1}{n-1} d
<\frac{7}{15}d.
$$
Since $\lambda<d/(10r^2)$, the above inequality also holds when $|T_i|<\lambda$. Therefore there are at least $\frac{8}{15}d$ nonnegative edges through $w_i$.
This completes the proof for $t\geq 2$, as the number of nonnegative edges through $w_1$ and $w_2$ is already at least
$\frac{16}{15}d-\lambda > d$ (when $n > 10r^3$).

If $t=1$, it means that $w_2<2rw_{2r}$. For $2 \le i \le 2r$, as before denote by $U_i$ the set of all the negative edges through $w_i$. Then similarly we define
$$\sum_{f \in U_i} \sum_{j \in f\backslash \{i\}} w_j = \lambda \sum_{j \ne i} \alpha_j w_j.$$
Note that $\alpha_i \in [0,1]$, $\sum_{j \ne i} \alpha_j w_j < -w_i |U_i|/\lambda$
and
$\sum_{j \ne i} (1-\alpha_j) = n-1- (r-1)|U_i|/\lambda$. So
\begin{align*}
-w_i &= \sum_{j \ne i}  w_j =\sum_{j \ne i} \alpha_j w_j + \sum_{j \ne i} (1-\alpha_j) w_j\le
-\frac{|U_i|}{\lambda} w_i+\sum_{i=1}^{2r}w_i+\sum_{j>2r} (1-\alpha_j) w_j\\
&\le -\frac{|U_i|}{\lambda} w_i + 1 + 2rw_2 + ((n-1-(r-1)|U_i|)/\lambda) w_i
\end{align*}
We have
$$\left(r\frac{|U_i|}{\lambda}-n\right)w_i \le 1 + 2rw_2 \le 2r^2w_{2r} + 2r \cdot 2rw_{2r} = 6r^2 w_{2r} \le 6r^2 w_i.$$
Therefore when $n>10r^3$,
$$|U_i| \le \frac{6r^2+n}{r} \cdot \frac{r-1}{n-1} \cdot d \le \Big(1-\frac{2}{5r}\Big)d.$$
Hence there are at least $\frac{2}{5r}d$ nonnegative edges through every $w_i$ when $2 \le i \le 2r$, together with the $\frac{8}{15}d$ nonnegative edges through $w_1$. Thus the total number of nonnegative edges is at least
$$\frac{8}{15}d + \frac{2}{5r}d (2r-1) - \binom{2r}{2} \lambda = \left( \frac{4}{3}-\frac{2}{5r}-\frac{2r^3-3r^2+r}{n-1}\right)d>d,$$
where we used that $\lambda=\frac{r-1}{n-1}d$ and $n>10r^3$.
\end{proof}
Combining Lemma \ref{lemma_case1}, Lemma \ref{lemma_case2} and Lemma \ref{lemma_case3} we show that $e^{+}(H) \ge d$. From the proofs it is not
hard to see that when $n>10r^3$ the only way to achieve the inequality is when the nonnegative edges form a star, i.e. contain a fixed vertex of $H$. This concludes the proof of Theorem \ref{main_theorem}. \hfill $\Box$

\vspace{0.15cm}
Next we use Theorem \ref{main_theorem} to prove the vector space analogue of Manickam-Mikl\'os-Singhi conjecture.

\vspace{0.15cm}
\noindent \textbf{Proof of Conjecture \ref{conj_mms_vector}}:
Let $H$ be the hypergraph such that the vertex set $V(H)$ consists of all the $1$-dimensional subspaces of $V=\mathbb{F}_q^n$. Obviously the number of vertices is equal to ${n \brack 1}_q$. Every $k$-dimensional subspace of $V$ defines an edge of $H$ which contains exactly ${k \brack 1}_q$ vertices. Therefore $H$ is an $r$-uniform hypergraph on $n'$ vertices with $r={k \brack 1}_q$ and $n'={n \brack 1}_q$. Since every two $1$-dimensional subspaces span a unique $2$-dimensional subspace, so the codegree of any two vertices in $H$ is equal to ${n-2 \brack k-2}_q$. Applying Theorem \ref{main_theorem}, as long as $n'>10r^3$, the minimum number of nonnegative edges in $H$ is at least equal to its degree, which is equal to ${n-1 \brack k-1}_q$. Actually the condition that $n'>10r^3$ is equivalent to
$$\frac{q^n-1}{q-1} > 10 \left( \frac{q^k-1}{q-1}\right)^3.$$
Since $n \ge 4k$, we have $q^n-1 \ge q^{4k}-1$. Moreover $(q^{4k}-1)/(q^k-1)^3 = (q^{3k}+q^{2k}+q^k+1)/(q^{2k}-2q^k+1) \ge q^k.$
From $k \ge 2$, we also have $(q-1)^2q^k \ge q^2(q-1)^2 > 10 $ if $q \ge 3$.
Therefore
$$\frac{q^{n}-1}{q-1} \ge \frac{(q^k-1)^3 q^k}{q-1} > 10 \left( \frac{q^k-1}{q-1}\right)^3.$$
For $q=2$, it is not hard to verify that the inequality is still satisfied when $k \ge 3$. The only remaining case is when $(q, k)= (2, 2)$. Again it is easy to check that the inequality holds when $n \ge 9$. The case $n=8$ was resolved by Manickam and Singhi \cite{manickam-singhi}, who proved
their conjecture when $k$ divides $n$.
\qed

\medskip
\noindent
{\bf Remark.} The statement of Conjecture \ref{conj_mms_vector} is known to be false only for $n < 2k$. Hence, it would be interesting to determine
the minimal $n=n(k)$ which implies this conjecture. Note that Theorem \ref{main_theorem} can be used to prove the assertion of Conjecture \ref{conj_mms_vector}
also for $n<4k$. For example it shows that this conjecture holds for $n \geq 3k$ and $ q \geq 5$, $n \geq 3k+1$ and $ q \geq 3$ or $n\geq 3k+2$ and all $q$. For large
$q$ the proof will work already starting with $n=3k-1$.

\section{The tightness of Theorem \ref{main_theorem}} \label{section_tightness}
In the previous section, we show that for every $r$-uniform $n$-vertex hypergraph with equal codegrees and $n>10r^3$, the minimum number of nonnegative edges is always achieved by the stars. Here we discuss the tightness of this result. As a warm-up example, recall that a finite projective plane has  $N^2+N+1$ points and $N^2+N+1$ lines such that every line contains $N+1$ points. Moreover every two points determine a unique line, and every two lines intersect at a unique point. If we regard points as vertices and lines as edges, this naturally corresponds to a $(N+1)$-uniform $(N+1)$-regular hypergraph with all codegrees equal to $1$. Let us assign weights $1$ to the $N+1$ points on a fixed line $l$, and weights $-\frac{N+1}{N^2}$ to the other points. Obviously the sum is nonnegative. On the other hand every line other than $l$ contains at most one point with positive weight, thus its sum of weight is at most $1-N \cdot \frac{N+1}{N^2} <0$. Therefore there is only one nonnegative edge. This already gives us a hypergraph with $n \sim r^2$ that is not MMS.

The next theorem provides an example of a hypergraph with $n\sim r^3$ for which there is a configuration of edges, different from a star, that also
achieves the minimum number of nonnegative edges.
\begin{theorem}
For infinitely many $r$ there is an $r$-uniform hypergraph $H$ with equal codegrees on $(r^3-2r^2+2r)$ vertices, and a weighting $w: V(H) \rightarrow \mathbb{R}$ with nonnegative sum, such that there are $\delta(H)$ nonnegative edges that do not form a star.
\end{theorem}
\begin{proof}
Let $r=q+1$, where $q$ is a prime power. Denote by $\mathbb{F}_q$ the finite field with $q$ elements. Define a hypergraph $H$ with the vertex set $V(H)$ consisting of points from the $3$-dimensional projective space $PG(3, \mathbb{F}_q)$. Here $PG(n, \mathbb{F}_q)=(\mathbb{F}_q^{n+1} \backslash \{0\}) /\sim$, with the equivalence relation $(x_0, \cdots , x_n) \sim (\sigma x_0, \cdots, \sigma x_n)$, where $\sigma$ is an arbitrary number from $\mathbb{F}_q$. It is easy to see that $n=|V(H)|=q^3+q^2+q+1=r^3-2r^2+2r$. Every $1$-dimensional subspace of $PG(3, \mathbb{F}_q)$ defines an edge of $H$ with $q+1=r$ elements. It is not hard to check that $H$ is $d$-regular for $d=q^2+q+1$, and every pair of vertices has codegree $1$.

Now we assign the weights to $V(H)$ in the following way. Let $S$ be the set of points of a $2$-dimensional projective subspace of $V(H)$, then $|S|=q^2+q+1$.
Every vertex from $S$ receives weight $1$, and every vertex outside $S$ has weight $-\frac{q^2+q+1}{q^3}$, such that the total weight is zero. Note that every edge has size $q+1$, so if it contains at most one vertex from $S$, its total weight is at most $1-q \cdot \frac{q^2+q+1}{q^3}<0$. Therefore every nonnegative edge must contain at least two vertices from $S$. Since $S$ is a subspace, the lines containing $2$ points from $S$ are completely contained in $S$.
There are precisely $q^2+q+1=d$ lines in $S$ (these are all the nonnegative edges in $H$) and they do not form a star.
\end{proof}

Finally, we give an example which shows that one might find hypergraphs with $n\sim r^3$ and weights such that the number of nonnegative edges is strictly smaller than the vertex degree.

\begin{theorem}
If $q$ and $q+1$ are both prime powers, then there exists a $(q+1)$-uniform $(q+1)^2$-regular
hypergraph $H$ on $(q^3+2q^2+q+1)$ vertices with all codegrees  equal to $1$, and an assignment of weights with nonnegative sum such that there are strictly less than $(q+1)^2$ nonnegative edges in $H$.
In particular if there are infinitely
many Mersenne primes, then we obtain infinitely many such hypergraphs.
\end{theorem}
\begin{proof}
Let $V(H)=V_1 \cup V_2$, such that $|V_1|=q^2+q+1$, and $|V_2|=q^2(q+1)$. We first take $H_1$ to be the projective plane $PG(2, \mathbb{F}_q)$ on $V_1$ with edges corresponding to the projective lines. In other words $H_1$ is a $(q+1)$-uniform hypergraph with degree $q+1$ and codegree $1$. The hypergraph $H_2$ consists of some $(q+1)$-tuples that intersect $V_1$ in exactly one vertex and intersect $V_2$ in $q$ vertices, such that $e(H_2)=(q^2+q+1)(q^2+q)$. $H_3$ is a $(q+1)$-uniform hypergraph on $V_2$ with $q^3$ edges. We will carefully define the edges of $H_2$ and $H_3$ soon.

We hope $H_2$ and $H_3$ to satisfy the following properties: (i) for every pair of vertices $u \in V_1$ and $v \in V_2$, their codegree in $H_2$ is equal to $1$; (ii) note that every edge in $H_2$ naturally induces a clique of size $q$ in $\binom{V_2}{2}$; while every edge in $H_3$ induces a clique of size $q+1$ in $\binom{V_2}{2}$. We hope these cliques form an edge partition of the complete graph $K_{|V_2|}=K_{q^2(q+1)}$. It is not hard to see that if (i), (ii) are both satisfied, then the hypergraph $H= H_1 \cup H_2 \cup H_3$ has codegree $\delta=1$. And $H$ is a regular hypergraph with degrees equal to
$$\delta \cdot \frac{n-1}{r-1} = \frac{(q^3+2q^2+q+1)-1}{(q+1)-1}=(q+1)^2.$$

Now we assign weights to $V(H)$, such that every vertex in $V_2$ receives a weight $-1$, while every vertex in $V_1$ receives a weight $\frac{q^2(q+1)}{q^2+q+1}$, so the total weight is zero. If an edge is nonnegative, it must contain at least two vertices from $V_1$, since $\frac{q^2(q+1)}{q^2+q+1} + (-1) \cdot q <0$. Such edge can only come from $H_1$. However we have $e(H_1) = q^2+q+1$, which is strictly smaller than the degree $(q+1)^2$. Therefore what remains is to show the existence of $H_2$ and $H_3$ satisfying (i), (ii). In other words, we need to find a clique partition in (ii) with
$$K_{q^2(q+1)} = q^3 \cdot  K_{q+1} \cup (q^2+q+1)(q^2+q) \cdot K_q. $$
Moreover, condition (i) requires that the family of $K_q$'s can be partitioned into $K_q$-factors.

A natural idea is to partition $[q^2(q+1)]=S_1 \cup \cdots \cup S_q$ with $|S_i|=q(q+1)$. Observe that the projective plane $PG(2, \mathbb{F}_q)$ defines a clique partition $K_{q^2+q+1}= (q^2+q+1) K_{q+1}$. By removing one vertex from it, we obtain a partition
$K_{q^2+q} = (q+1) \cdot K_q \cup q^2 \cdot K_{q+1}$. By doing this for every $S_i$, we get the $q^3$ copies of $K_{q+1}$ we want, and $(q+1)q$ copies of $K_q$, which clearly form a $K_q$-factor, since the $(q+1)$ copies of $K_q$ from each $S_i$ are pairwise disjoint.
We still need to find an edge partition of the balanced complete $q$-partite graph $K_{q^2+q, \cdots, q^2+q}$ into $(q^2+q)^2$ copies of $K_q$, so that they also can be grouped into $q^2+q$ disjoint $K_q$-factors.

Suppose we know that $q+1$ is also a prime power. Label the vertices in $K_{q^2+q, \cdots, q^2+q}$ by $(x, y, z)$ where $x \in \mathbb{F}_{q}$, $y \in \mathbb{F}_{q}$, and $z \in \mathbb{F}_{q+1}$. Two vertices $(x, y, z)$ and $(x', y', z')$ are adjacent iff $x \ne x'$. Now we define $(q^2+q)^2$ cliques $C_{i, j, k, l}$'s for $i, k \in \mathbb{F}_q$ and $j, l \in \mathbb{F}_{q+1}$. The clique $C_{i,j,k,l}$ consists of vertices in the form of $(x, i+kx, j+lf(x))$ for all $x \in \mathbb{F}_q$, where $f$ is a fixed injective map from $\mathbb{F}_q$ to $\mathbb{F}_{q+1}$. Suppose $(x, y, z)$ and $(x', y', z')$ with $x \ne x'$ are both contained in the clique $C_{i,j,k,l}$, then we have
\begin{align*}
i+kx=y~&~~i+kx'=y'\\
j+lf(x)  = z ~&~~ j+lf(x') = z'.
\end{align*}
Since $x \ne x'$, the first two equations uniquely determine $i$ and $k$. Moreover, $f(x)$ and $f(x')$ are different elements of $\mathbb{F}_{q+1}$ since $f$ is injective, thus $j, l$ are also uniquely determined. Therefore $\{C_{i,j,k,l}\}$ forms a $K_q$-partition of the edges of $K_{q^2+q, \cdots, q^2+q}$, and it is not hard to see that they can be partitioned into $K_q$-factors by fixing $k$ and $l$.

By the above discussions, if we have both of $q$ and $q+1$ to be powers of prime, in particular when $q=2^n-1$ is a Mersenne prime, then we can explicitly construct the hypergraph.
\end{proof}

\section{Two additional generalizations} \label{section_kint}
In the next two subsections we discuss generalizations of the two Manickam-Mikl\'os-Singhi conjectures and prove
Theorem \ref{theorem_k-tuple} and Theorem \ref{theorem_k-tuple_vector}. For $k=1$ they follow from Theorem \ref{main_theorem}, thus we can assume
that $k\geq 2$. In that case, as we mentioned earlier in the introduction, these two theorems are not direct consequences of Theorem \ref{main_theorem} because the codegrees in the corresponding hypergraphs are not equal.

\subsection{Generalization of MMS} \label{subsection_mms}
In this subsection we will prove Theorem \ref{theorem_k-tuple}. This requires some new ideas and techniques since
direct adaptation of the proof of Theorem \ref{main_theorem} does not work. Indeed, it is easy to construct a weighting such that there is no nonnegative edge through the vertex (a $k$-set) of maximal weight. For example say $k=2$, one can take $w_{\{1, 2\}}=1$, the weights of all the $2n-4$ pairs containing $1$ or $2$ to be $-\frac{n-3}{10}$, and the rest to have weights roughly $\frac{2}{5}$. For sufficiently large $n$, no $t$-set containing $\{1, 2\}$ has nonnegative total weights.

First we prove a simple lemma from linear algebra.
\begin{lemma}\label{lemma_gauss}
Suppose the $s \times s$ lower triangular matrix $\beta=\{\beta_{i,j}\}$ satisfies that $\beta_{i,i} >0$ and for every $j<k \le i$, $0 \le \beta_{i,j} \le \beta_{i,k}$. Then for a given vector $\vec{b}=(b_1, \cdots, b_s)$ such that $b_1 \ge \cdots \ge b_s \ge 0$, the equation $\vec{b} = \vec{\gamma} \cdot \beta$ has a unique solution $\vec{\gamma}=(\gamma_1, \cdots, \gamma_s)$ and moreover $0 \le \gamma_i \le b_i/\beta_{i, i}$.
\end{lemma}
\begin{proof}
The existence and uniqueness of $\vec{\gamma}$ follow from the fact that $\beta$ is invertible. Next we inductively prove $0 \le \gamma_i \le b_i/\beta_{i, i}$. We start
from $\gamma_s$, from the equation we know $b_s=\gamma_s \beta_{s, s}$. So $\gamma_s = b_s / \beta_{s,s}$ and the inductive hypothesis is true.
Suppose $0 \le \gamma_i \le b_i/\beta_{i,i}$ for every $i>k$. Now from the linear equation, we have
$b_k= \beta_{k, k} \gamma_k + \beta_{k+1, k} \gamma_{k+1} \cdots + \beta_{s, k} \gamma_s.$ Since $\gamma_i, i >k$ and $\beta_{i,j}$ are nonnegative, we have
$\gamma_k \le b_k / \beta_{k,k}$.
Note that $\beta_{i,j}$ is increasing in $j$, so for every $k+1 \le i \le s$, $\beta_{i, k} \le \beta_{i, k+1}$. Therefore $b_k \le \beta_{k, k} \gamma_k + \sum_{i=k+1}^s \beta_{i, k+1} \gamma_i = \beta_{k,k} \gamma_k + b_{k+1}$. Since $0 \le b_{k+1} \le b_{k}$, we know that $\gamma_k \geq 0$.
\end{proof}

Without loss of generality, we may assume that $\sum_{X \subset \binom{ [n]}{k}} w_{X} = 0$; and
$w_{\{1, \cdots, k\}}$, or alternatively written as $w_{[k]}$, has the largest positive weight. We may let $w_{[k]}=1$, then $w_X \le 1$ for every
$k$-set $X$.
Throughout this section, we also assume that $n>Ct^{3k+3}$, here $C$ is some sufficiently large constant.
The next lemma shows that if the sum of weights of certain edges is very negative, then we already have enough nonnegative edges.
\begin{lemma}\label{very_negative}
If for some subset $|L|=k$,
$$\sum_{L \subset Y, |Y|=t} \sum_{L \ne X \subset Y} w_X \le -\frac{1}{13t^{2k}}\binom{n-k}{t-k},$$
and $\sum_{X \ne L} w_X \ge -1$, then there are more than $\binom{n-k}{t-k}$ nonnegative edges in $H$.
\end{lemma}
\begin{proof}
We may rewrite the left hand side of the inequality as
\begin{align*}
&\binom{n-2k}{t-2k} \sum_{|X \cap L|=0} w_X + \binom{n-2k+1}{t-2k+1} \sum_{|X \cap L|=1} w_X + \cdots + \binom{n-k-1}{t-k-1} \sum_{|X \cap L|=k-1} w_X \\
&=\binom{n-k-1}{t-k-1} \sum_{|X \cap L| \le k-1} w_X - \sum_{j=0}^{k-2} \Big( b_j \cdot \sum_{|X \cap L|= j} w_X\Big).
\end{align*}
Here we let $b_j= \binom{n-k+1}{t-k+1}-\binom{n-2k+j}{t-2k+j}$. Note that $\sum_{|X \cap L| \le k-1} w_X=\sum_{X \ne L} w_X \ge -1$. Since $n>Ct^{3k+3}$, this implies
\begin{align} \label{cite_inequality}
\sum_{j=0}^{k-2} \Big(b_j \cdot \sum_{|X \cap L| = j} w_X\Big) &\ge \frac{1}{13t^{2k}} \binom{n-k}{t-k} - \binom{n-k-1}{t-k-1} \ge  \frac{1}{14t^{2k}} \binom{n-k}{t-k}.
\end{align}

For a fixed integer $0 \le y \le k-1$, denote by $D_y$ the number of nonnegative $t$-sets $Z$ with $|Z \cap L|=y$. If $D_y > \binom{n-k}{t-k}$ then we are done. Otherwise assume $D_y \le \binom{n-k}{t-k}$ for every $y$. We estimate the following sum:
$$\sum_{|Z \cap L|=y, |Z|=t} \sum_{X \subset Z} w_X .$$
Since every nonnegative $t$-set contributes to the sum at most $\binom{t}{k}$, it is at most
$\binom{t}{k} D_y \le \binom{t}{k}\binom{n-k}{t-k}$. By double counting, the above sum also equals
$\sum_{j=0}^y ( \beta_{y, j} \cdot \sum_{|X \cap L|=j} w_X ),$
where $\beta_{y, j}=\binom{k-j}{y-j}\binom{n-2k+j}{t-k-y+j}$, note that $\beta_{y, j}=0$ when $j<k+y-t$. When $j \ge k+y-t$, since $n \gg t$, for fixed $y$,
$\beta_{y, j}$ is increasing in $j$. Also note that $b_j$ is decreasing in $j$.
Let $\vec{\gamma}=(\gamma_0, \cdots, \gamma_{k-2})$ be the unique solution of the system
of equations $\vec{b} = \vec{\gamma} \cdot \beta$, then by Lemma \ref{lemma_gauss}
\begin{eqnarray*}
\sum_{j=0}^{k-2} \Big(b_j \cdot \sum_{|X \cap L| = j} w_X\Big) &=&
\sum_{j=0}^{k-2} \sum_{y=j}^{k-2} \beta_{y, j} \gamma_y\sum_{|X \cap L| = j} w_X
= \sum_{y=0}^{k-2} \gamma_y \cdot \sum_{j=0}^y \Big( \beta_{y, j} \cdot \sum_{|X \cap L|=j} w_X \Big)\\
 &\le& \binom{t}{k}\binom{n-k}{t-k}\sum_{y=0}^{k-2}  \gamma_y \leq
\binom{t}{k}\binom{n-k}{t-k}\sum_{y=0}^{k-2} \frac{b_y}{\beta_{y,y}} .
\end{eqnarray*}
Since $b_y/\beta_{y,y}=\left(\binom{n-k-1}{t-k-1}-\binom{n-2k+y}{t-2k+y}\right)/\binom{n-2k+y}{t-k} \le \binom{n-k-1}{t-k-1}/\binom{n-2k}{t-k}.$ We have
\begin{align*}
\sum_{j=0}^{k-2} \Big(b_j \cdot \sum_{|X \cap L| = j} w_X\Big) \le \binom{t}{k}\binom{n-k}{t-k} \cdot(k-1)\cdot \frac{\binom{n-k-1}{t-k-1}}{\binom{n-2k}{t-k}} \leq \frac{t^{k+1}}{n}\binom{n-k}{t-k}.
\end{align*}
For $n>Ct^{3k+3}$ this contradicts \eqref{cite_inequality}.
\end{proof}

We now assume that the $t^k$-th largest weight in $H$ is $w_P$ and consider several cases.
\begin{lemma}\label{lemma_largeweight}
If $w_P > \frac{1}{t^{2k}}$, there are more than $\binom{n-k}{t-k}$ nonnegative edges in the hypergraph $H$.
\end{lemma}
\begin{proof}
We will show that every vertex whose weight is larger than $w_{P}$ is contained in at least $\frac{3}{2t^k} \binom{n-k}{t-k}$ nonnegative edges, otherwise there are already $>\binom{n-k}{t-k}$ nonnegative edges.
For simplicity we just need to prove this statement for $w_{P}$ itself. Suppose there are $S$ negative edges containing $w_{P}$, which are denoted by $e_1, \cdots, e_{S}$ (as $t$-subsets). And $e_{S+1}, \cdots, e_{\binom{n-k}{t-k}}$ are the other (thus nonnnegative) edges containing $w_P$. We have
\begin{align}
\sum_{i=1}^{\binom{n-k}{t-k}} \sum_{P \ne X \subset e_i} w_{X} &= \sum_{i=1}^{S} \sum_{P \ne X \subset e_i} w_X + \sum_{i=S+1}^{\binom{n-k}{t-k}} \sum_{P \ne X \subset e_i} w_{X}\nonumber\\
&\le -w_P \cdot S + w_P \cdot \left(\binom{n-k}{t-k}-S\right) \cdot \left(\binom{t}{k}-1\right) + t^k \cdot \binom{n-k-1}{t-k-1} \label{sum_estimate}
\end{align}
Here we used that there are at most $t^k$ sets $X$ whose weight is larger than $w_{P}$ (but always $\leq 1$), and the number of times every such set appears in the sum is at most $\binom{n-k-1}{t-k-1}$. If $S \ge (1-\frac{3}{2t^k}) \binom{n-k}{t-k}$, then the above expression is at most
\begin{align*}
&- \binom{n-k}{t-k} \left(\Big(1-\frac{3}{2t^k}\Big) w_P - \frac{3}{2t^k} \cdot \binom{t}{k} \cdot w_P - \frac{t^{k+1}}{n} \right),
\end{align*}
which can be further bounded by
$$- \binom{n-k}{t-k} \left( \Big(1-\frac{3}{2t^k}-\frac{3}{2 \cdot k!}\Big)w_P -\frac{t^{k+1}}{n}\right)
<- \binom{n-k}{t-k} \cdot \frac{1}{13}w_P \le -\binom{n-k}{t-k} \cdot \frac{1}{13t^{2k}}.$$
The last inequality uses that $n>Ct^{3k+3}$, $t>k \ge 2$ and therefore $1-\frac{3}{2t^k}-\frac{3}{2 \cdot k!} \ge 1-\frac{3}{2 \cdot 3^2} - \frac{3}{2 \cdot 2!}=\frac{1}{12}$. Since we also have $\sum_{X \ne P} w_X = -w_P \ge -1$,
Lemma \ref{very_negative} for $L=P$ immediately gives $>\binom{n-k}{t-k}$ nonnegative edges.

Therefore we can assume that for the $t^k$ sets with largest weights, the number of nonnegative edges containing each such set is at least $\frac{3}{2t^k} \binom{n-k}{t-k}$. Using the union bound, the number of nonnegative edges is at least
$$t^k \cdot \frac{3}{2t^k} \binom{n-k}{t-k} - \binom{t^k}{2} \binom{n-k-1}{t-k-1} ,$$
which is also larger than $\binom{n-k}{t-k}$.
\end{proof}

The next lemma covers the case when $w_P$ is smaller than $\frac{1}{t^{2k}}$, and there are significant number of  negative edges containing $\{1, \cdots, k\}$.

\begin{lemma}\label{lemma_smallweight}
If the $t^k$-th largest weight $w_{P}$ is smaller than $1/t^{2k}$, and there are less than $(1-\frac{1}{t^k}) \binom{n-k}{t-k}$ nonnegative edges containing $\{1, \cdots, k\}$, then there are at least $\binom{n-k}{t-k}$ nonnegative edges in $H$.
\end{lemma}
\begin{proof}
We consider all the $t$-tuples containing $\{1, \cdots, k\}$, similarly as before suppose there are $S \ge \frac{1}{t^k}\binom{n-k}{t-k}$ negative edges $e_1, \cdots, e_S$ and nonnegative edges $e_{S+1}, \cdots, e_{\binom{n-k}{t-k}}$, we get
\begin{align*}
\sum_{[k] \subset Z, |Z|=t} \sum_{X \subset Z, X \ne [k]} w_X &= \sum_{i=1}^{S} \sum_{X \subset e_i, X \ne [k]} w_X + \sum_{i=S+1}^{\binom{n-k}{t-k}} \sum_{X \subset e_i, X \ne [k]} w_X\\
&\le - S + \frac{1}{t^{2k}} \cdot \left(\binom{n-k}{t-k}-S\right) \cdot \left(\binom{t}{k}-1\right) + t^k \cdot \binom{n-k-1}{t-k-1}\\
& \le -\left(S - \frac{1}{k! \cdot t^{k}} \left(\binom{n-k}{t-k}-S \right) - t^k \binom{n-k-1}{t-k-1}\right)
\end{align*}
The first inequality is by bounding the $t^k$ largest weights in the second sum by $1$ and the rest by $\frac{1}{t^{2k}}$. It also
uses the fact that two sets are contained in at most $\binom{n-k-1}{t-k-1}$ edges.
Since $S \ge \frac{1}{t^k}\binom{n-k}{t-k}$, we have
$$\sum_{[k] \subset Z, |Z|=t} \sum_{X \subset Z, X \ne [k]} w_X \le -\left(\frac{1}{2t^k} \binom{n-k}{t-k} - t^k \binom{n-k-1}{t-k-1}\right)
\leq - \left(\frac{1}{2t^k} -\frac{t^{k+1}}{n} \right)\binom{n-k}{t-k}.$$
For large $n$ the right hand side is at most
$-\frac{1}{3t^k} \binom{n-k}{t-k}$. We also have
$\sum_{X \ne [k]} w_X = -w_{[k]}=-1$.
Now we once again can apply Lemma \ref{very_negative} for $L=\{1, \cdots, k\}$ to show the existence of $>\binom{n-k}{t-k}$ nonnegative edges.
\end{proof}

It remains to prove the case when $\{1, \cdots, k\}$ is contained in at least $(1-\frac{1}{t^k}) \binom{n-k}{t-k}$ nonnegative edges.
\begin{lemma}\label{lemma_last}
If $\{1, \cdots, k\}$ is contained in at least $(1-\frac{1}{t^k}) \binom{n-k}{t-k}$ nonnegative edges, then there are at least $\binom{n-k}{t-k}$ nonnegative edges in $H$.
\end{lemma}
\begin{proof}
Note that if every edge containing $\{1, \cdots, k\}$ is nonnegative, this already gives $\binom{n-k}{t-k}$ nonnegative edges and the lemma is proved. So we may assume that there is a negative edge $f$ (as $t$-subset) through $\{1, \cdots, k\}$ with $\sum_{X \subset f} w_X <0$. Suppose the largest weight outside the edge $f$ is $w_Q$, where $|Q \cap f| \le k-1$.
Now we define new weights $w'$, such that
\begin{align*}
w'_X=
\begin{cases}
-\binom{t}{k}   & \text{if } X \subset f \\
w_X/w_Q        & \text{otherwise.}
\end{cases}
\end{align*}
Then for every $X \not \subset f$, $w'_X \le 1$ and $w'_Q=1$. Now we consider all the $\binom{n-k}{t-k}$ $t$-tuples containing the $k$-set $Q$. As usual, assume that $S$ of them has negative sum according to $w'$. If $S \ge (1-\frac{3}{2t^k})\binom{n-k}{t-k}$, we have the following estimate:
\begin{align*}
\sum_{Q \subset Y, |Y|=t} \sum_{X \subset Y, X \ne Q} w'_X & \le -S + \left(\binom{n-k}{t-k}-S \right) \cdot \binom{t}{k} \le -\binom{n-k}{t-k} \cdot \left(1-\frac{3}{2t^k}-\frac{3}{2t^k}\binom{t}{k}\right)\\
& \le -\binom{n-k}{t-k} \cdot \left(1-\frac{3}{2t^k}-\frac{3}{2k!}\right) \le - \frac{1}{12}\binom{n-k}{t-k}.
\end{align*}
Note that since $\sum_{X \subset f} w_X <0$, we have
\begin{align*}
\sum_{X \ne Q} w'_X &= \sum_{X \ne Q} w_X/w_Q - \sum_{X \subset f} \left(w_X/w_Q+\binom{t}{k}\right) \ge -1 - \binom{t}{k}^2 \ge -t^{2k}.
\end{align*}
If we apply lemma \ref{very_negative} for $L=Q$ and the weight $w''_X=w'_X/t^{2k}$, we get $>\binom{n-k}{t-k}$ nonnegative edges for the new weight function $w'$. Note that every such nonnegative edge can not share with $f$ a common $k$-subset, otherwise its total weight is at most $(\binom{t}{k}-1) - \binom{t}{k} < 0$. Hence these nonnegative edges are also nonnegative edges for the original weight function $w$.

By the above discussion, it remains to consider the case $S<(1-\frac{3}{2t^k})\binom{n-k}{t-k}$. Then there are at least $\frac{3}{2t^k} \binom{n-k}{t-k}$ nonnegative edges containing $Q$, together with the $(1-\frac{1}{t^k})\binom{n-k}{t-k}$ nonnegative edges containing $\{1, \cdots,  k\}$. Since $\{1, \cdots, k \}$ and $w_Q$ have codegree at most $\binom{n-k-1}{t-k-1} < \frac{1}{2t^k} \binom{n-k}{t-k}$, we have in total more than $\binom{n-k}{t-k}$ nonnegative edges.
\end{proof}

\subsection{Generalization of vector MMS} \label{subsection_vectormms}
Our techniques from the previous section also allow us to prove a generalization of the vector space version of Manickam-Mikl\'os-Singhi conjecture.
Since the proof of this result is very similar to that of Theorem \ref{theorem_k-tuple} we only state the appropriate variants of
the lemmas involved. The detailed proofs of these lemmas can be found in the appendix of this paper. The proof of Theorem \ref{theorem_k-tuple_vector} follows immediately from combining these lemmas.
As before, we define the hypergraph $H$ to have the vertex set ${V \brack k}$ and every edge corresponds to a $t$-dimensional subspace. It is easy to check that the hypergraph is ${t \brack k}_q$-uniform on ${n \brack k}_q$ vertices. Like the previous section, we also assume that $[k]$ is the $k$-dimensional subspace with $w_{[k]}=1$ and for every $X$, $w_X \le 1$. All the following lemmas are proven under the assumption that $n>C(t-k)k$ for sufficiently large constant $C$.


\begin{lemma}\label{very_negative_vector}
If for some $k$-dimensional subspace $L$,
$$\sum_{L \subset Y, Y \in {V \brack t}} \sum_{L \ne X \subset Y} w_X \le -\frac{1}{24 {t \brack k}_q^2} {n-k \brack t-k}_q,$$
and $\sum_{X \ne L} w_X \ge -1$, then there are more than ${n-k \brack t-k}_q$ nonnegative edges in $H$.
\end{lemma}

We now assume that the $3 {t \brack k}_q$-th largest weight in $H$ is $w_P$, and consider the following several cases.
\begin{lemma}\label{lemma_largeweight_vector}
If $w_P > 1/(4{t \brack k}_q^2)$, there are more than ${n-k \brack t-k}_q$ nonnegative edges in $H$.
\end{lemma}

\begin{lemma}\label{lemma_smallweight_vector}
If $w_{P} \le 1/(4{t \brack k}_q^2)$, and there are less than $(1-\frac{1}{2{t \brack k}_q}) {n-k \brack t-k}_q$ nonnegative edges containing $[k]$, then there are at least ${n-k \brack t-k}_q$ nonnegative edges in $H$.
\end{lemma}

\begin{lemma}\label{lemma_last_vector}
If $[k]$ is contained in at least $(1-\frac{1}{2{t \brack k}_q}){n-k \brack t-k}_q$ nonnegative edges, then there are at least ${n-k \brack t-k}_q$ nonnegative edges in $H$.
\end{lemma}

\section{Concluding Remarks}
A $r \textrm{--}(n,t,\lambda)$ block design is a collection of $t$-subsets of $[n]$ such that every $r$ elements are contained in exactly $\lambda$ subsets.
In \cite{rands}, Rands proved the following generalization of Erd\H os-Ko-Rado theorem:
given a $r \textrm{--}(n,t,\lambda)$ block design $H$ and $0<s<r$, then there exists a function $f(t,r,s)$ such that
if $H$ has an $s$-intersecting subhypergraph $H'$, then if $n>f(t,r,s)$, the number of edges in $H'$ is at most $b_s$, which is the
number of blocks through $s$ vertices. Note that Erd\H os-Ko-Rado theorem corresponds to the very special case when $H=\binom{[n]}{t}$ and $s=1$.
Moreover, when $(s, r)=(1, 2)$, this is an analogue of our Theorem \ref{main_theorem}, and when the block design is complete, it is similar to
Theorem \ref{theorem_k-tuple}. Using tools developed in the previous section, we can prove the following generalization of
Manickam-Mikl\'os-Singhi cojecture to designs. Given an $r \textrm{--}(n, t, \lambda)$ design $H$,
for $j=1, \cdots, t$, let $d_j$ be the number of blocks containing a fixed set of $j$ elements. Obviously $d_r=\lambda$, and by double counting,
$d_j = \frac{\binom{n-j}{r-j}}{\binom{t-j}{r-j}} \lambda$.
\begin{theorem}
Let $k, r, t$ be positive integers with $t \ge r \ge 2k$, $n>Ct^{3k+3}$ for sufficiently large $C$ and let
$\{w_{X}\}_{X \in \binom{[n]}{k}}$ be a weight assignment with $\sum_{X \in \binom{ [n]}{k}} w_{X} \ge 0$. Then for a given $r \textrm{--}(n, t, \lambda)$ design $H$,
 the number of blocks $B$ with $\sum_{X \subset B, X \in \binom{[n]}{k}} w_X \ge 0$ is at least $d_k=  \frac{\binom{n-k}{r-k}}{\binom{t-k}{r-k}} \lambda$.
\end{theorem}
It would be interesting if one can remove the condition $t \ge r \ge 2k$ in this statement. This will give a general result unifying our Theorems
\ref{main_theorem} and \ref{theorem_k-tuple}. The only additional ingredient needed to prove the above theorem is the following fact.
For two disjoint vertex subsets $|A|=a$ and $|B|=b$ of a $r \textrm{--}(n, t, \lambda)$ design, the number of edges containing every vertex
from $A$ while not containing
any vertex in $B$ is equal to $\frac{\binom{n-r-b}{t-r}}{\binom{n-r}{t-r}} \frac{\binom{n-a-b}{r-a}}{\binom{t-a}{r-a}} \cdot \lambda$.
We will omit any further details here and will return to this problem in the future.


In Section \ref{section_tightness}, we gave an example of infinitely many $r$-uniform $n$-vertex hypergraphs with equal
codegrees and $n\sim r^3$ not having the MMS property, based on the assumption that there are infinitely many Mersenne primes.
Since the largest known Mersenne number has more than ten million digits, our example already gives (unconditionally) a huge hypergraph
with $n$ cubic in $r$. Still it would be interesting to construct infinitely many such hypergraphs directly, without relying on the existence of Mersenne primes?

In Section \ref{section_kint}, we proved two additional generalizations of the Manickam-Mikl\'os-Singhi conjecture. Both results can be regarded as the 
analogues of the Erd\H os-Ko-Rado theorem on the $k$-intersecting families for sufficiently large $n$. It would be interesting to determining the exact range 
for which these theorems hold. For example when $k=1$, Theorem \ref{theorem_k-tuple} only gives $n>t^6$ while we know from \cite{pokrovskiy} that it is 
true already for $n$  linear in $t$.

\vspace{0.2cm}
\noindent
{\bf Acknowledgment.}
We would like to thank Ameerah Chowdhury for bringing to our attention a Manikam-Miklos-Sighi conjecture for vector spaces and for sharing with us her preprint
on this topic.

\appendix
\section{Missing proofs from Section \ref{subsection_vectormms}}
Throughout this section we use that for $a>b$, $ q^{(a-b)b} \leq {a \brack b}_q \leq q^{(a-b)b+b}$.
\vspace{0.1cm}

\noindent \textbf{Proof of Lemma \ref{very_negative_vector}}:
We may rewrite the left hand side of the inequality as
\begin{align*}
&{n-2k \brack t-2k}_q \sum_{\dim (X \cap L)=0} w_X + {n-2k+1 \brack t-2k+1}_q \sum_{\dim (X \cap L)=1} w_X + \cdots + {n-k-1 \brack t-k-1}_q \sum_{\dim (X \cap L)=k-1} w_X \\
&={n-k-1 \brack t-k-1}_q \sum_{\dim (X \cap L) \le k-1} w_X - \sum_{j=0}^{k-2} \Big( b_j \cdot \sum_{\dim (X \cap L)= j} w_X\Big).
\end{align*}
Here we let $b_j= {n-k+1 \brack t-k+1}_q-{n-2k+j \brack t-2k+j}_q$. Note that $\sum_{\dim (X \cap L) \le k-1} w_X=\sum_{X \ne L} w_X \ge -1$. Since $n>Ck(t-k)$, this implies
\begin{align} \label{cite_inequality_vector}
\sum_{j=0}^{k-2} \Big(b_j \cdot \sum_{|X \cap L| = j} w_X\Big) &\ge \frac{1}{24{t\brack k}_q^2} {n-k \brack t-k}_q - {n-k-1 \brack t-k-1}_q \ge  \frac{1}{25{t \brack k}_q^2} {n-k \brack t-k}_q.
\end{align}

For a fixed integer $0 \le y \le k-1$, denote by $D_y$ the number of nonnegative $t$-dimensional subspaces $Z$ with $\dim (Z \cap L)=y$. If $D_y > {n-k \brack t-k}_q$ then we are done. Otherwise assume $D_y \le {n-k \brack t-k}_q$ for every $y$. We estimate the following sum:
$$\sum_{\dim (Z \cap L)=y, \dim Z=t} \sum_{X \subset Z} w_X.$$
Since every nonnegative $t$-dimensional subspace contributes to the sum at most ${t \brack k}_q$, it is at most
${t \brack k}_q D_y \le {t \brack k}_q {n-k \brack t-k}_q$. By double counting, the above sum also equals
$\sum_{j=0}^y ( \beta_{y, j} \cdot \sum_{\dim (X \cap L) = j} w_X ).$ Here for a $k$-dimensional subspace $X$ with $\dim (X \cap L)=j$,
$\beta_{y, j}$ denotes the number of $t$-dimensional subspaces $Z$ such that $X \subset Z$ and $\dim (Z \cap L) = y$. There are
$\frac{(q^k-q^j) \cdots (q^k-q^{y-1})}{(q^y-q^j) \cdots (q^y-q^{y-1})}$ ways to extend $X \cap L$ to $Z \cap L$.
Let $Q= \textrm{span} \{X, Z \cap L\}$, and $R = \textrm{span} \{X, L\}$. Then $\dim Q=k+y-j$, $\dim R= 2k-j$, and $Q \subset R$.
The next step is to extend $Q$ to $Z$ such that $Z \cap R  = Q$. The number of ways is equal to
$\frac{(q^{n}-q^{2k-j}) \cdots (q^{n}-q^{t+k-y-1})}{(q^{t}-q^{k+y-j}) \cdots (q^{t}-q^{t-1})}$. Note that this is only nonzero for $j \ge k+y-t$, in this case $\beta_{y, j}$ is the product of these two expressions, which is roughly $q^{(k-y)(y-j)+(n-t)(t-k+j-y)}$. Since $t-k+j-y \geq 0$, it is increasing in $j$ for large $n$. Also note that $b_j$ is decreasing in $j$.
Let $\vec{\gamma}=(\gamma_0, \cdots, \gamma_{k-2})$ be the unique solution of the system
of equations $\vec{b} = \vec{\gamma} \cdot \beta$, then by Lemma \ref{lemma_gauss}
\begin{eqnarray*}
\sum_{j=0}^{k-2} \Big(b_j \cdot \sum_{\dim (X \cap L) = j} w_X\Big) &=&
\sum_{j=0}^{k-2} \sum_{y=j}^{k-2} \beta_{y, j} \gamma_y \sum_{\dim (X \cap L)=j} w_X=
\sum_{y=0}^{k-2} \gamma_y \cdot \sum_{j=0}^y \Big( \beta_{y, j} \cdot \sum_{\dim (X \cap L)=j} w_X \Big) \\
&\le& {t \brack k}_q {n-k \brack t-k}_q\sum_{y=0}^{k-2} \gamma_y
\leq {t \brack k}_q {n-k \brack t-k}_q\sum_{y=0}^{k-2} \frac{b_y}{\beta_{y,y}}.
\end{eqnarray*}
It is easy to check that $\beta_{y,y} \geq q^{(n-t)(t-k)}$ and so $b_y/\beta_{y,y}=({n-k-1 \brack t-k-1}_q-{n-2k+y \brack t-2k+y}_q)/q^{(n-t)(t-k)} \le {n-k-1 \brack t-k-1}_q/q^{(n-t)(t-k)}$.   Therefore
\begin{align*}
\sum_{j=0}^{k-2} \Big(b_j \cdot \sum_{|X \cap L| = j} w_X\Big) \le {t \brack k}_q {n-k \brack t-k}_q \cdot(k-1)\cdot \frac{{n-k-1 \brack t-k-1}_q}{q^{(n-t)(t-k)}} \leq (k-1){t \brack k}_q q^{t-k-(n-t)} {n-k \brack t-k}_q,
\end{align*}
that for $n>Ck(t-k)$ contradicts \eqref{cite_inequality_vector}. \qed
~\\
~\\
\medskip
\noindent \textbf{Proof of Lemma \ref{lemma_largeweight_vector}}:
We will show that every $k$-subspace whose weight is larger than $w_{P}$ is contained in at least $\frac{1}{2{t \brack k}_q} {n-k \brack t-k}_q$ nonnegative edges, otherwise
there are already more than ${n-k \brack t-k}_q$ nonnegative edges.
For simplicity we just need to prove this statement for $w_{P}$ itself. Suppose there are $S$ negative edges containing $w_{P}$, which are denoted by
$e_1, \cdots, e_{S}$ (as $t$-dimensional subspaces). And $e_{S+1}, \cdots, e_{{n-k \brack t-k}_q}$ are the other (thus nonnnegative) edges
containing $w_P$. We have
\begin{align}
\sum_{i=1}^{{n-k \brack t-k}_q} \sum_{P \ne X \subset e_i} w_{X} &= \sum_{i=1}^{S} \sum_{P \ne X \subset e_i} w_X + \sum_{i=S+1}^{{n-k \brack t-k}_q} \sum_{P \ne X \subset e_i} w_{X}\nonumber\\
&\le -w_P \cdot S + w_P \cdot \left({n-k \brack t-k}_q-S\right) \cdot \left({t \brack k}_q-1\right) + 3 {t \brack k}_q \cdot {n-k-1 \brack t-k-1}_q \label{sum_estimate_vector}
\end{align}
Here we used that there are at most $3 {t \brack k}_q$ vertices $X$ whose weight is larger than $w_P$ (but always $\le 1$), and the number of
times every such weight appear in the sum is at most ${n-k-1 \brack t-k-1}_q$. If $S \ge \Big(1-\frac{1}{2{t \brack k}_q}\Big) {n-k \brack t-k}_q$,
then the above expression is at most
\begin{align*}
&- {n-k \brack t-k}_q \left( \Big(1-\frac{1}{2{t \brack k}_q}\Big) w_P - \frac{1}{2{t \brack k}_q} \cdot {t \brack k}_q \cdot w_P - 3 {t \brack k}_q \cdot \frac{q^{t-k}-1}{q^{n-k}-1} \right),
\end{align*}
which can be further bounded by
$$- {n-k \brack t-k}_q \left( \Big(\frac{1}{2}-\frac{1}{2{t \brack k}_q}\Big)w_P -3 {t \brack k}_q \cdot \frac{q^{t-k}-1}{q^{n-k}-1}\right)
<- {n-k \brack t-k}_q \cdot \frac{1}{3}w_P \le -{n-k \brack t-k}_q \cdot \frac{1}{12{t \brack k}_q^2}.$$
The first inequality is because  $t>k \ge 2$ and $q \ge 2$, so ${t \brack k}_q \ge 7$, and also because $n>Ck(t-k)$ for large $C$. Since we also have
$\sum_{X \ne P} w_X = -w_P \ge -1.$
Lemma \ref{very_negative_vector} for $L=P$ immediately gives $>{n-k \brack t-k}_q$ nonnegative edges.

Therefore we can assume that for the $3 {t \brack k}_q$ vertices with largest weights, the number of nonnegative edges containing each such vertex is at least $\frac{1}{2 {t \brack k}_q} {n-k \brack t-k}_q$. Using the union bound, the number of nonnegative edges is at least
$$3 {t \brack k}_q \cdot \frac{1}{2 {t \brack k}_q} {n-k \brack t-k}_q   - \binom{3 {t \brack k}_q}{2} {n-k-1 \brack t-k-1}_q
\geq \frac{3}{2}\left(1-\frac{3{t \brack k}^2_q}{q^{n-t}}\right){n-k \brack t-k}_q ,$$
which is also larger than ${n-k \brack t-k}_q$ when $n >Ck(t-k)$.
\qed
~\\
~\\
\medskip
\noindent \textbf{Proof of Lemma \ref{lemma_smallweight_vector}}:
We consider all the $t$-dimensional subspaces containing $[k]$, similarly as before suppose there are $S \ge \frac{1}{2{t \brack k}_q} {n-k \brack t-k}_q$ negative edges $e_1, \cdots, e_S$ and nonnegative edges $e_{S+1}, \cdots, e_{{n-k \brack t-k}_q}$, we get
\begin{align*}
\sum_{[k] \subset Z, \dim Z=t} \sum_{X \subset Z, X \ne [k]} w_X &= \sum_{i=1}^{S} \sum_{X \subset e_i, X \ne [k]} w_X + \sum_{i=S+1}^{{n-k \brack t-k}_q} \sum_{X \subset e_i, X \ne [k]} w_X\\
&\le - S + \frac{1}{4{t \brack k}_q^2} \cdot \left({n-k \brack t-k}_q-S\right) \cdot \left({t \brack k}_q-1\right) + 3{t \brack k}_q \cdot {n-k-1 \brack t-k-1}_q\\
& \le -{n-k \brack t-k}_q\left(\frac{S}{{n-k \brack t-k}_q} - \frac{1}{4{t \brack k}_q}  - 3{t \brack k}_q \cdot \frac{q^{t-k}-1}{q^{n-k}-1}\right)\\
& \le -{n-k \brack t-k}_q\left(\frac{1}{2{t\brack k}_q} - \frac{1}{4{t \brack k}_q}- 3{t \brack k}_q  q^{-(n-t)}\right)
\end{align*}
The first inequality is by bounding the $3{t \brack k}_q$ largest weights in the second sum by $1$ and the rest by $\frac{1}{4{t \brack k}_q^2}$. It also
uses the fact that two $k$-dimensional subspaces are contained in at most ${n-k-1 \brack t-k-1}_q$ $t$-dimensional subspaces.
For $n>Ck(t-k)$, we have
$$\sum_{[k] \subset Z, \dim Z=t} \sum_{X \subset Z, X \ne [k]} w_X \le -{n-k \brack t-k}_q \cdot \frac{1}{5{t\brack k}_q}.$$
We also have $\sum_{X \ne [k]} w_X = -w_{[k]}=-1.$ Now we once again can apply Lemma \ref{very_negative_vector} for $L=[k]$ to show the existence of $>{n-k \brack t-k}_q$ nonnegative edges.
\qed
~\\
~\\
\medskip
\noindent \textbf{Proof of Lemma \ref{lemma_last_vector}}:
Note that if every $t$-dimensional subspaces containing $[k]$ is nonnegative, this already gives ${n-k \brack t-k}_q$ nonnegative edges and the lemma is proved. So we may assume that there is a negative edge $f$ (as $t$-dimensional subspace) containing $[k]$ with $\sum_{X \subset f} w_X <0$. Suppose the largest weight outside the edge $f$ is $w_Q$, where $\dim (Q \cap f) \le k-1$.
Now we define new weights $w'$, such that
\begin{align*}
w'_X=
\begin{cases}
-{t \brack k}_q   & \text{if } X \subset f \\
w_X/w_Q        & \text{otherwise.}
\end{cases}
\end{align*}
Then for every $X \not \subset f$, $w'_X \le 1$ and $w'_Q=1$. Now we consider all the ${n-k \brack t-k}_q$ $t$-dimensional subspaces containing $Q$.
As usual, assume that $S$ of them has negative sum according to $w'$. If $S \ge \Big(1-\frac{2}{3{t \brack k}_q}\Big){n-k \brack t-k}_q$, we have the
following estimate:
\begin{eqnarray*}
\sum_{Q \subset Y, \dim Y=t} \sum_{X \subset Y, X \ne Q} w'_X &\le& -S +
\left({n-k \brack t-k}_q-S \right) \cdot {t \brack k}_q \le -{n-k \brack t-k}_q \cdot \left(1-\frac{2}{3{t \brack k}_q}-\frac{2}{3}\right) \\
&\le& - \frac{1}{12}{n-k \brack t-k}_q.
\end{eqnarray*}
Note that since $\sum_{X \subset f} w_X <0$, we have
\begin{align*}
\sum_{X \ne Q} w'_X &= \sum_{X \ne Q} w_X/w_Q - \sum_{X \subset f} \left(w_X/w_Q+ {t\brack k}_q\right) \ge -1 - {t \brack k}_q^2.
\end{align*}
If we apply Lemma \ref{very_negative_vector} for $L=Q$ and the new weighting $w''_X=w'_X/(2{t\brack k}_q^2)$, we get $>{n-k \brack t-k}_q$ nonnegative edges for weight $w'$. Note that every such nonnegative edge
cannot share with $f$ a common $k$-dimensional subspace, otherwise its total weight is at most $({t \brack k}_q-1) - {t \brack k}_q < 0$. Hence these nonnegative edges are also nonnegative edges for the original weighting $w$.

By the above discussion, it remains to consider the case $S<\Big(1-\frac{2}{3{t \brack k}_q}\Big){n-k \brack t-k}_q$.
Then there are at least $\frac{2}{3{t \brack k}_q} {n-k \brack t-k}_q$ nonnegative edges containing $Q$, together
with the $\Big(1-\frac{1}{2{t \brack k}_q}\Big) {n-k \brack t-k}_q$ nonnegative edges containing $[k]$. Since $[k]$ and $w_Q$ have
codegree at most ${n-k-1 \brack t-k-1}_q \le \frac{1}{6{t \brack k}_q} {n-k \brack t-k}_q$, we have in total more than ${n-k \brack t-k}_q$ nonnegative edges.
\qed
\end{document}